\documentclass[10pt]{article}

\usepackage[
  paperwidth=165mm,
  paperheight=240mm,
  top=18mm,
  bottom=19mm,
  left=17mm,
  right=17mm,
  footskip=10mm
]{geometry}
\usepackage[T1]{fontenc}
\usepackage{lmodern}
\usepackage{amsmath,amssymb,amsthm,mathtools}
\usepackage[english]{babel}
\usepackage{microtype}
\usepackage[hidelinks]{hyperref}

\newtheorem{theorem}{Theorem}[section]
\newtheorem{proposition}[theorem]{Proposition}
\newtheorem{lemma}[theorem]{Lemma}
\newtheorem{corollary}[theorem]{Corollary}

\theoremstyle{definition}
\newtheorem{definition}[theorem]{Definition}
\newtheorem{example}[theorem]{Example}
\newtheorem{remark}[theorem]{Remark}

\newcommand{\F}{\mathbb F}

\newcommand{\A}{\mathcal A}
\newcommand{\B}{\mathcal B}
\DeclareMathOperator{\ord}{ord}

\title{\texorpdfstring{$F$}{F}-sets of arbitrary finite width}
\author{Alessandro Giannoni}
\date{}
\hypersetup{
  pdftitle={F-sets of arbitrary finite width},
  pdfauthor={Alessandro Giannoni}
}

\begin{document}

\maketitle

\begin{abstract}
Ferraguti and Micheli conjectured that non-trivial $F$-sets of every
prescribed width exist over each finite field.  We prove the finite-width
part of their conjecture over $\F_q$ for every $q\neq2,3$.  Fix a suitable
prime $\ell$.  For a degree bound $D$, we consider the saturated family of
all irreducible polynomials $g(X^{\ell^j})$ whose core $g$ has degree at most
$D$.  A factor-descent lemma shows that every irreducible factor arising from
a shifted difference belongs to the same family and has strictly smaller
core degree.  The saturated family is therefore an $F$-set of finite width.
Dirichlet's theorem for $\F_q[X]$ produces core ladders of arbitrary length,
and Kummer lifting reproduces each ladder at infinitely many scales.  These
parallel ladders force the width to be sufficiently large; an appropriate
tail of the nullity filtration then has the prescribed width.

\end{abstract}

\section{Introduction}

Throughout the paper, $q$ is a prime power and $I_q$ denotes the set of monic
nonconstant irreducible polynomials in $\F_q[X]$.  Following Andrade, Miller,
Pratt and Trinh~\cite{AndradeMillerPrattTrinh2014}, an $F$-set is a subset
$A\subseteq I_q$ such that, for every $f\in A$, all monic irreducible divisors
of $f(X)-f(0)$ belong to $A$.  An $F$-set is called \emph{non-trivial} if it
is different from $I_q$.  Thus non-trivial means proper, not necessarily
infinite; whenever infinitude is needed we state it separately.  For example,
$\{X\}$ and $I_q$ are both $F$-sets, the first non-trivial and the second not.

Ferraguti and Micheli~\cite{FerragutiMicheli2016} introduced the width of an
$F$-set.  Their constructions give examples of width one for $q\neq2,3$, and
they proved that examples of width two exist for every $q\neq2,3$; see
\cite[Sections~2 and~3 and Theorem~5.1(a)]{FerragutiMicheli2016}.  They
conjectured the existence of non-trivial $F$-sets of arbitrary width over
every finite field.  The width is defined by repeatedly removing the
irreducibles that have no successor in the set, where $g$ is a successor of
$f$ if $f$ divides $g(X)-g(0)$.  Our main result proves the finite-width part
of their conjecture, with the only possible exceptions $\F_2$ and $\F_3$.

\begin{theorem}\label{thm:main}
Let $q\neq2,3$ be a prime power.  For every integer $r\geq1$, there exists an
infinite, non-trivial $F$-set in $\F_q[X]$ of width exactly $r$.
\end{theorem}

Our proof combines Dirichlet's theorem for $\F_q[X]$ with Kummer lifting.  Its
main technical ingredient is a factor-descent statement for power
substitutions.  This makes it possible to place arbitrarily long ladders of
irreducibles inside a saturated $F$-set whose width is finite and controlled
by the degrees of its cores.  Infinitely many Kummer lifts of the ladders give
the required lower bound, and a suitable tail of the nullity filtration has
width exactly $r$.  The restriction $q\neq2,3$ enters through the existence
of a prime satisfying the Kummer hypotheses; positive finite width over
$\F_2$ and $\F_3$ remains open.

\section{\texorpdfstring{$F$}{F}-sets, successors and width}

We begin with the relation measured by width.  If $f,g\in I_q$, we call $g$ a
\emph{successor} of $f$ when
$$
 f\mid g(X)-g(0).
$$
In these terms the $F$-set condition is a downward-closure condition: together
with $g$, an $F$-set must contain every irreducible polynomial having $g$ as a
successor.

\begin{definition}\label{def:nullity}
Let $A$ be an $F$-set.  Its \emph{nullity} is
$$
 N(A)=\{f\in A:f\nmid g(X)-g(0)\text{ for every }g\in A\}.
$$
Thus $N(A)$ is the set of elements of $A$ having no successor inside $A$.
\end{definition}

Every nonempty $F$-set contains $X$.  Indeed, if $f\in A$, then
$X\mid f(X)-f(0)$, so closure gives $X\in A$.  Moreover, $X$ is its own
successor, because $X\mid X-X(0)$; consequently $X\notin N(A)$.  This is why
width is defined using the first finite term rather than the first empty term.

As observed in \cite[Section~3]{FerragutiMicheli2016}, the set
$A\setminus N(A)$ is again an $F$-set.  Thus the deletion can be repeated.

\begin{definition}\label{def:width}
The \emph{nullity filtration} of an $F$-set $A$ is
$$
 A^{(0)}=A,\qquad
 A^{(n+1)}=A^{(n)}\setminus N\bigl(A^{(n)}\bigr).
$$
The \emph{width} $w(A)$ is the least $n\geq0$ for which $A^{(n)}$ is
finite, if such an $n$ exists; otherwise $w(A)=\infty$.
\end{definition}

This is the definition of Ferraguti and Micheli
\cite{FerragutiMicheli2016}.  In particular, an $F$-set has width zero exactly
when it is finite.  If $A$ is infinite and $N(A)=\varnothing$, then the
filtration is stationary and $w(A)=\infty$.  A finite set may also have empty
nullity: width records the first finite stage, not whether that stage can be
reduced any further.

The following two elementary observations will be used repeatedly.

\begin{lemma}\label{lem:tails}
For every $s,n\geq0$,
\begin{equation}\label{eq:tail}
 \bigl(A^{(s)}\bigr)^{(n)}=A^{(s+n)}.
\end{equation}
In particular, if $w(A)=W<\infty$, then
\begin{equation}\label{eq:tail-width}
 w\bigl(A^{(s)}\bigr)=W-s
 \qquad(0\leq s\leq W).
\end{equation}
\end{lemma}

\begin{proof}
The first identity follows by induction on $n$.  The second follows from the
minimality of $W$: the terms $A^{(m)}$ are infinite for $m<W$, while
$A^{(W)}$ is finite.
\end{proof}

\begin{lemma}\label{lem:rank-bound}
Let $A$ be an $F$-set and suppose that there is a map
$$
 \rho:A\longrightarrow\{0,1,\ldots,D\}
$$
such that $\rho(X)=0$ whenever $X\in A$, the fibre $\rho^{-1}(0)$ is finite,
and, whenever $Q\neq X$ and $P\mid Q(X)-Q(0)$, one has
$\rho(P)<\rho(Q)$.  Then
$$
 A^{(n)}\subseteq\{P\in A:\rho(P)\leq D-n\}
 \qquad(0\leq n\leq D),
$$
and hence $w(A)\leq D$.
\end{lemma}

\begin{proof}
The assertion is clear for $n=0$.  Suppose it holds for some $n<D$.  If
$P\in A^{(n)}$ has $\rho(P)=D-n>0$, then $P\neq X$.  Moreover, $X$ cannot be
a successor of $P$, since $P\mid X$ would force $P=X$.  Hence every successor
$Q$ of $P$ in $A^{(n)}$ satisfies $Q\neq X$ and therefore
$\rho(Q)>\rho(P)=D-n$, which is impossible by the induction hypothesis.
Thus every element of rank $D-n$ belongs to $N(A^{(n)})$, proving the
inductive step.  At $n=D$, only the finite set $\rho^{-1}(0)$ can remain.
\end{proof}

\section{Kummer lifting and factor descent}

We now record the finite-field facts used in the construction.  For a prime
$p$ and a positive integer $m$, write $v_p(m)$ for the $p$-adic valuation of
$m$.  We also write $\ord(\beta)$ for the multiplicative order of a nonzero
algebraic element $\beta$ over a finite field.  We shall use the standard
degree-order relation
$$
 [\F_Q(\beta):\F_Q]
 =\min\{n\geq1:\ord(\beta)\mid Q^n-1\};
$$
see, for example, \cite[Theorem~3.3]{LidlNiederreiter1997}.  Recall also that
$\F_q^*$ is cyclic of order $q-1$.  Thus, if $\ell\mid q-1$, then $\F_q$
contains a primitive $\ell$-th root of unity and non-$\ell$-th-powers exist.

We isolate the primes used in the power substitutions.

\begin{definition}\label{def:admissible}
A prime $\ell$ is \emph{$q$-admissible} if $\ell\mid q-1$ and either
$\ell$ is odd or $q\equiv1\pmod4$.
\end{definition}

The extra clause matters only for $\ell=2$ and packages the hypotheses needed
both in the Kummer criterion and in the valuation formula below.

\begin{lemma}\label{lem:admissible-exists}
If $q\neq2,3$, then there exists a $q$-admissible prime.
\end{lemma}

\begin{proof}
If $q>2$ is even, then every prime divisor of the odd integer $q-1$ is odd
and hence $q$-admissible.  If $q\equiv1\pmod4$, then $\ell=2$ is
$q$-admissible.  Finally, suppose that $q>3$ and $q\equiv3\pmod4$.  The
integer $(q-1)/2$ is odd and greater than one, so any prime divisor of it is
odd, divides $q-1$, and is therefore $q$-admissible.
\end{proof}

We shall also use the following special case of lifting the exponent.

\begin{lemma}\label{lem:valuation}
Let $\ell$ be $q$-admissible.  For all integers $d,n\geq1$,
$$ v_\ell\bigl(q^{dn}-1\bigr)
 =v_\ell\bigl(q^d-1\bigr)+v_\ell(n).
$$\end{lemma}

\begin{proof}
Put $u=q^d$ and write $n=\ell^k r$, where $(r,\ell)=1$.  Since
$u\equiv1\pmod\ell$, the factorization
$$ u^r-1=(u-1)(1+u+\cdots+u^{r-1})
$$shows that
$$ v_\ell(u^r-1)=v_\ell(u-1),
$$because the second factor is congruent to $r$ modulo $\ell$.

Suppose first that $\ell$ is odd.  If $x\equiv1\pmod\ell$ and
$x=1+\ell^a c$ with $a\geq1$ and $\ell\nmid c$, the binomial expansion of
$x^\ell-1$ has first term $\ell^{a+1}c$, while every remaining term is
divisible by $\ell^{a+2}$.  Hence
$$ v_\ell(x^\ell-1)=v_\ell(x-1)+1.
$$Starting with $x=u^r$ and iterating this identity $k$ times gives
$$ v_\ell(u^n-1)=v_\ell(u-1)+k
 =v_\ell(u-1)+v_\ell(n).
$$
Now suppose that $\ell=2$.  Admissibility gives $u\equiv1\pmod4$, and since
$r$ is odd, also $u^r\equiv1\pmod4$.  For every $x\equiv1\pmod4$,
$$ v_2(x^2-1)=v_2(x-1)+v_2(x+1)=v_2(x-1)+1.
$$Iterating this identity $k=v_2(n)$ times, again starting with $x=u^r$,
proves the same formula for $\ell=2$.
\end{proof}

The following is the specialization of
\cite[Proposition~4.1]{FerragutiMicheli2016} to $K=\F_q$ and $p=\ell$.
The definition of $q$-admissibility supplies both the required roots of unity
and, when $\ell=2$, the condition that $-1$ be a square in $\F_q$.

\begin{lemma}\label{lem:kummer}
Let $\ell$ be $q$-admissible and let $g\in\F_q[X]$ be monic and
irreducible.  If $g(0)$ is not an $\ell$-th power in $\F_q$, then
$$ g\bigl(X^{\ell^j}\bigr)
$$is irreducible for every $j\geq0$.
\end{lemma}

The following factor-descent statement is a special case of known
factorization results for power substitutions; see Butler~\cite{Butler1955}
and Graner~\cite[Theorems~3 and~14]{Graner2026}.  We give an elementary proof
adapted to the bounded-core construction.  It shows that every irreducible
factor of the relevant power substitution has a presentation with a core of
the same degree as $a$.

\begin{lemma}\label{lem:descent}
Let $\ell$ be $q$-admissible.  Let $a\in\F_q[X]$ be monic irreducible,
$a\neq X$, and put $d=\deg a$.  If a monic irreducible polynomial $h$ divides
$a(X^{\ell^j})$, then
$$ h(X)=b\bigl(X^{\ell^t}\bigr)
$$for some integer $0\leq t\leq j$ and some monic irreducible polynomial
$b\in\F_q[X]$ of degree $d$.
\end{lemma}

\begin{proof}
Let $\beta$ be a root of $h$ and put $\alpha=\beta^{\ell^j}$.  Then
$\alpha$ is a root of $a$.  Since $a\neq X$, both $\alpha$ and $\beta$ are
nonzero, and
$$ K=\F_q(\alpha)=\F_{q^d}.
$$Write
$$ \ord(\beta)=\ell^s m,\qquad (m,\ell)=1,
$$and set
$$ E=v_\ell(q^d-1),\qquad t=\max\{0,s-E\}.
$$Now
$$ \ord(\alpha)=\ell^{\max\{s-j,0\}}m.
$$This order divides $q^d-1$, because $\alpha\in K$.  Thus
$m\mid q^d-1$ and $\max\{s-j,0\}\leq E$.  If $s>E$, the latter inequality
gives $s-E\leq j$; if $s\leq E$, then $t=0$.  Hence $t\leq j$ in both cases.

By the degree-order relation above, applied with $Q=q^d$, the degree
$[K(\beta):K]$ is the least positive integer $n$ for which
$\ord(\beta)$ divides $q^{dn}-1$.  The prime-to-$\ell$ part $m$ already
divides $q^d-1$, while Lemma~\ref{lem:valuation} shows that the $\ell$-part
divides precisely when $v_\ell(n)\geq t$.  The least such $n$ is $\ell^t$.
On the other hand, $\gamma=\beta^{\ell^t}$ has order dividing $q^d-1$, and
hence $\gamma\in K$.  Moreover,
$$ \alpha=\gamma^{\ell^{j-t}},
$$so
$K=\F_q(\alpha)\subseteq\F_q(\gamma)\subseteq K$ and therefore
$\F_q(\gamma)=K$.  Let $b$ be the minimal polynomial of $\gamma$ over
$\F_q$.  Since $\gamma\neq0$, we have $b\neq X$.  Thus $b$ is monic
irreducible of degree $d$, and
$b(X^{\ell^t})$ has $\beta$ as a root.  Since
$K\subseteq\F_q(\beta)$, the tower law and the preceding degree computation
give
$$ \deg h=[\F_q(\beta):\F_q]
 =d\ell^t=\deg b(X^{\ell^t}).
$$Since both polynomials are monic and $h$ divides $b(X^{\ell^t})$, they are
equal.
\end{proof}

\begin{remark}\label{rem:F3-obstruction}
The admissibility hypothesis is essential for this uniform descent.  Over
$\F_3$, for example, take $a(X)=X+1$ and $j=2$.  Then
$$
 a(X^{2^2})=X^4+1
 =(X^2+X+2)(X^2+2X+2).
$$
Both quadratic factors are irreducible over $\F_3$, but neither has the form
$b(X^{2^t})$ with $\deg b=1$.  Thus the conclusion of
Lemma~\ref{lem:descent} can fail when $2$ is not $q$-admissible.  The
even-degree alternative in the Kummer criterion may still produce individual
irreducible towers over $\F_3$; what fails here is the uniform factor descent
needed for the whole bounded-core saturation.
\end{remark}

\section{The bounded-core construction}

Fix a $q$-admissible prime $\ell$ and an integer $D\geq1$.  We think of
$g(X^{\ell^j})$ as a polynomial with \emph{core} $g$ and \emph{scale} $j$.
The bounded-core saturation consists of every irreducible power substitution
whose core has degree at most $D$:
\begin{equation}\label{eq:AD}
\begin{split}
 \A_D(\ell)=\{X\}\cup\bigl\{g(X^{\ell^j}):{}&g\in\F_q[X]
 \text{ monic irreducible},\ g\neq X,\\
 &1\leq\deg g\leq D,\ j\geq0,\
 g(X^{\ell^j})\text{ irreducible}\bigr\}.
\end{split}
\end{equation}
For $P\in\A_D(\ell)$, $P\neq X$, define its \emph{core degree} by
\begin{equation}\label{eq:rho}
 \rho(P)=\min\{\deg g:P=g(X^{\ell^j})\text{ for some }g,j
 \text{ occurring in \eqref{eq:AD}}\},
\end{equation}
and put $\rho(X)=0$.

The same polynomial may a priori admit more than one presentation as a power
substitution.  Taking the minimum in \eqref{eq:rho} makes $\rho(P)$ intrinsic;
no uniqueness of the core is being assumed.

\begin{proposition}\label{prop:bounded-core}
Let $Q\in\A_D(\ell)$ with $Q\neq X$.  Every monic irreducible factor $P$ of
$Q(X)-Q(0)$ belongs to $\A_D(\ell)$ and satisfies
$$
 \rho(P)<\rho(Q).
$$
In particular, $\A_D(\ell)$ is an $F$-set.
\end{proposition}

\begin{proof}
Choose a representation of minimal core degree,
$$
 Q(X)=g(X^{\ell^j})
$$
with $\deg g=\rho(Q)$.  Since $g\neq X$ is irreducible, $g(0)\neq0$.
Factor
$$
 g(Y)-g(0)=Y^v\prod_{s=1}^u a_s(Y)^{e_s},
$$
where $v\geq1$ and the $a_s$ are monic irreducibles different from $Y$; the
product is allowed to be empty.  Every $a_s$ has degree strictly smaller than
$\deg g$, because
$$
 \sum_{s=1}^u e_s\deg a_s=\deg g-v\leq\deg g-1.
$$
After substituting $Y=X^{\ell^j}$, we obtain
$$
 Q(X)-Q(0)=X^{v\ell^j}
 \prod_{s=1}^u a_s(X^{\ell^j})^{e_s}.
$$
If $P=X$, then $P\in\A_D(\ell)$ and
$\rho(P)=0<\rho(Q)$.  Otherwise, $P$ divides
$a_s(X^{\ell^j})$ for some $s$.  Lemma~\ref{lem:descent} gives
$$
 P(X)=b(X^{\ell^t})
$$
for some $0\leq t\leq j$ and some monic irreducible $b$ with
$$
 \deg b=\deg a_s<\deg g=\rho(Q).
$$
Since $P$ is irreducible, this is an admissible presentation in
\eqref{eq:AD}.  Thus $P\in\A_D(\ell)$ and
$\rho(P)\leq\deg b<\rho(Q)$.

For $Q=X$, the only irreducible factor of $Q(X)-Q(0)=X$ is already in the
set.  Hence $\A_D(\ell)$ is an $F$-set.
\end{proof}

\begin{corollary}\label{cor:bounded-core}
The set $\A_D(\ell)$ is an infinite, non-trivial $F$-set and
$$
 w\bigl(\A_D(\ell)\bigr)\leq D.
$$
More precisely,
$$
 \A_D(\ell)^{(n)}
 \subseteq\{P\in\A_D(\ell):\rho(P)\leq D-n\}
 \qquad(0\leq n\leq D).
$$
\end{corollary}

\begin{proof}
Proposition~\ref{prop:bounded-core} and Lemma~\ref{lem:rank-bound} apply to
the core-degree map, whose zero fibre is $\{X\}$.  This proves the asserted
bound on the filtration and on the width.

Choose $c\in\F_q^*$ which is not an $\ell$-th power.  By
Lemma~\ref{lem:kummer}, all the polynomials
$$
 X^{\ell^j}+c\qquad(j\geq0)
$$
are irreducible and belong to $\A_D(\ell)$.  They have distinct degrees, so
$\A_D(\ell)$ is infinite.

Finally, choose a prime number $n>\max\{D,\ell\}$.  By the standard
existence theorem for irreducible polynomials over finite fields, there is a
monic irreducible polynomial of degree $n$ over $\F_q$.  No such polynomial
belongs to $\A_D(\ell)$: the degree of every element different from $X$ is
$d\ell^j$ with $d\leq D$, whereas $n$ is prime, $n>D$, and $n\neq\ell$.
Therefore $\A_D(\ell)\neq I_q$.
\end{proof}

\begin{remark}\label{rem:saturation}
The saturation ranges over all irreducible cores of degree at most $D$, not
only over the cores belonging to the ladder constructed below.  This is what
absorbs the additional factors introduced by Dirichlet's theorem.  The price
is that the exact width of $\A_D(\ell)$ is not read directly from its
definition.  We shall only need the finite upper bound and a lower bound
provided by parallel ladders; a tail of the nullity filtration will recover
the exact prescribed width.
\end{remark}

\section{Dirichlet ladders and exact width}

We state precisely the second external input used in the paper.

\begin{theorem}\label{thm:dirichlet}
Let $A,M\in\F_q[X]$, with $M$ nonconstant and $\gcd(A,M)=1$.  For every
integer $N\geq0$, there is a monic irreducible polynomial $P$ of degree greater
than $N$ such that
$$
 P\equiv A\pmod M.
$$
In fact, there are infinitely many such polynomials.
\end{theorem}

This is the function-field analogue of Dirichlet's theorem on primes in
arithmetic progressions; see, for example, Kornblum and
Landau~\cite{KornblumLandau1919}.  We need only the existence assertion and
the freedom to make the degree large.

A sequence $f_1,\ldots,f_r$ of irreducibles will be called a \emph{ladder} if
$f_{i+1}$ is a successor of $f_i$ for $1\leq i<r$.  A ladder forces its lower
members to survive successive nullity deletions.  The one-step Dirichlet
construction of successors already appears in
\cite[Example~3.3]{FerragutiMicheli2016}; here we iterate it while keeping a
fixed non-$\ell$-th-power constant term, and then repeat the resulting ladder
at infinitely many Kummer scales.

\begin{lemma}\label{lem:ladder}
Let $\ell$ be $q$-admissible and let $r\geq1$.  There are monic irreducible
polynomials $g_1,\ldots,g_r$ and an element $c\in\F_q^*$ such that
\begin{enumerate}
\item $c$ is not an $\ell$-th power and $g_i(0)=c$ for every $i$;
\item $\deg g_1<\cdots<\deg g_r$;
\item $g_i\mid g_{i+1}-g_{i+1}(0)$ for $1\leq i<r$;
\item $g_i(X^{\ell^k})$ is irreducible for every $i$ and every $k\geq0$.
\end{enumerate}
\end{lemma}

\begin{proof}
Choose a non-$\ell$-th-power $c\in\F_q^*$ and put $g_1=X+c$.  Such a $c$
exists by cyclicity of $\F_q^*$.  Suppose that $g_i$ has been chosen.  Since
$c$ is a nonzero constant, $\gcd(c,Xg_i)=1$.  By
Theorem~\ref{thm:dirichlet}, one may choose a monic irreducible $g_{i+1}$, of
degree larger than $\deg g_i$, such that
$$
 g_{i+1}\equiv c\pmod{Xg_i}.
$$
Because $Xg_i$ divides $g_{i+1}-c$, evaluation at zero gives
$g_{i+1}(0)=c$, while $g_i\mid g_{i+1}-c$.  The last assertion follows from
Lemma~\ref{lem:kummer}.
\end{proof}

\begin{remark}\label{rem:ladder-not-fset}
The congruence used in the proof has the form
$$
 g_{i+1}(X)=c+Xg_i(X)h_i(X)
$$
for some $h_i\in\F_q[X]$.  After Kummer substitution,
$$
 g_{i+1}(X^{\ell^k})-c
 =X^{\ell^k}g_i(X^{\ell^k})h_i(X^{\ell^k}).
$$
Thus the lifted ladder contains the required predecessor
$g_i(X^{\ell^k})$, but it need not contain the irreducible factors of
$h_i(X^{\ell^k})$.  The ladder by itself is therefore not generally an
$F$-set.  The bounded-core family $\A_D(\ell)$ is the ambient $F$-set that
incorporates all these additional factors.
\end{remark}

\begin{theorem}\label{thm:admissible}
Suppose that $q$ admits a $q$-admissible prime.  For every integer $r\geq1$
there exists an infinite, non-trivial $F$-set in $\F_q[X]$ of width exactly
$r$.
\end{theorem}

\begin{proof}
Fix a $q$-admissible prime $\ell$ and choose a ladder
$g_1,\ldots,g_r$ as in Lemma~\ref{lem:ladder}.  Put $D=\deg g_r$ and
$$
 G_{i,k}(X)=g_i(X^{\ell^k})
 \qquad(1\leq i\leq r,\ k\geq0).
$$
All these polynomials belong to $\A_D(\ell)$, and
\begin{equation}\label{eq:lifted-ladder}
 G_{i,k}\mid G_{i+1,k}-G_{i+1,k}(0)
 \qquad(1\leq i<r).
\end{equation}

Let $\A_D(\ell)^{(n)}$ be the nullity filtration.  We claim that
\begin{equation}\label{eq:survival}
 G_{i,k}\in\A_D(\ell)^{(n)}
 \quad\text{whenever}\quad 1\leq i\leq r-n.
\end{equation}
This is clear for $n=0$.  If it holds for $n$, then
$G_{i+1,k}\in\A_D(\ell)^{(n)}$ whenever $i\leq r-n-1$; by
\eqref{eq:lifted-ladder}, $G_{i,k}$ is not in the nullity of that set and
survives the next step.  This proves \eqref{eq:survival} by induction.

In particular, all the polynomials
$$
 G_{1,k}=X^{\ell^k}+c\qquad(k\geq0)
$$
belong to $\A_D(\ell)^{(r-1)}$.  They have distinct degrees, so that term is
infinite.  If
$$
 W=w\bigl(\A_D(\ell)\bigr),
$$
Corollary~\ref{cor:bounded-core} gives $W\leq D$, while the preceding
paragraph shows that the $(r-1)$-st term is infinite and hence that
$W\geq r$.  Define
$$
 \B_r=\A_D(\ell)^{(W-r)}.
$$
The set $\B_r$ is an $F$-set.  It is infinite because, by minimality of $W$,
every filtration term with index smaller than $W$ is infinite.  It is
non-trivial because $\B_r\subseteq\A_D(\ell)\neq I_q$, and
Lemma~\ref{lem:tails} gives
$$
 w(\B_r)=W-(W-r)=r.
$$
\end{proof}

\begin{example}\label{ex:width-three}
The first case beyond the previously known widths can be made completely
explicit.  Work over $\F_{13}$, take $\ell=2$ and $c=8$, and set
$$
 g_1=X+8,\qquad
 g_2=X^2+8X+8,\qquad
 g_3=X^3+8X^2+8X+8.
$$
The element $8$ is not a square in $\F_{13}$.  The discriminant of $g_2$ is
$6$, also a nonsquare, and direct evaluation shows that $g_3$ has no root in
$\F_{13}$.  Hence all three polynomials are irreducible.  Moreover,
$$
 g_2-8=Xg_1,\qquad g_3-8=Xg_2.
$$
Since $2$ is $13$-admissible, Lemma~\ref{lem:kummer} lifts this ladder at
every scale.  It follows that $\A_3(2)^{(2)}$ is infinite, so
$w(\A_3(2))\geq3$.  Corollary~\ref{cor:bounded-core} gives the reverse
inequality, and therefore
$$
 w\bigl(\A_3(2)\bigr)=3.
$$
Thus, in this example, the saturated bounded-core family itself already has
the desired width and no passage to a tail is required.
\end{example}

\begin{remark}\label{rem:existential-tail}
The bounded-core set $\A_D(\ell)$ is explicit once the ladder and $D$ are
fixed, but the argument only places its width $W$ in the interval
$r\leq W\leq D$.  Accordingly, the tail $\B_r$ is an existence construction:
the proof does not determine the integer $W-r$ or list the elements removed
before that stage.  No such determination is needed for the exact-width
statement.
\end{remark}

\begin{proof}[Proof of Theorem~\ref{thm:main}]
By Lemma~\ref{lem:admissible-exists}, every $q\neq2,3$ admits a
$q$-admissible prime.  The conclusion follows from
Theorem~\ref{thm:admissible}.
\end{proof}

\begin{remark}
Independently of the theorem, the finite set $\{X\}$ is a non-trivial
$F$-set of width zero.  The argument does not decide whether every positive
finite width occurs over $\F_2$ and $\F_3$.  Those are the only remaining
fields for the finite-width problem.
\end{remark}

\end{document}